\documentclass[onefignum,onetabnum]{siamart171218}
\usepackage{braket}
\usepackage{color}
\usepackage{ amssymb }

\definecolor{darkred}{rgb}{.7,0,0}

\definecolor{darkgreen}{rgb}{0.3,0.6,0.0}

\definecolor{darkblue}{rgb}{0,0,0.7}

\renewcommand{\mathfrak}{\mathsf}

\newcommand{\mmu}{\nu}
\newcommand{\cG}{\mathcal{G}}
\newcommand{\cF}{\mathcal{F}}

\usepackage{lipsum}
\usepackage{amsfonts}
\usepackage{graphicx}
\usepackage{epstopdf}
\usepackage{algorithmic}
\ifpdf
  \DeclareGraphicsExtensions{.eps,.pdf,.png,.jpg}
\else
  \DeclareGraphicsExtensions{.eps}
\fi

\newsiamremark{remark}{Remark}
\newsiamremark{hypothesis}{Hypothesis}
\crefname{hypothesis}{Hypothesis}{Hypotheses}
\newsiamthm{claim}{Claim}

\headers{Diffusive Optical Tomography in the Bayesian Framework}{Kit Newton, Qin Li, and Andrew Stuart}

\title{Diffusive Optical Tomography in the Bayesian Framework\thanks{Submitted to the editors DATE.
\funding{This work was funded by AFOSR Grant FA9550-17-1-0185, NSF DMS 1619778 and 1750488, and NSF TRIPODS 1740707.}}}

\author{Kit Newton\thanks{Department of Mathematics, University of Wisconsin-Madison
  (\email{kcnewton@math.wisc.edu}, \url{http://kitcnewton.github.io}).}
\and Qin Li\thanks{Department of Mathematics, University of Wisconsin-Madison
  (\email{qinli@math.wisc.edu}).}
\and Andrew M. Stuart\thanks{Computing $+$ Mathematical Sciences, California Institute of Technology
  (\email{astuart@caltech.edu}).}}

\usepackage{amsopn}

\ifpdf
\hypersetup{
  pdftitle={Diffusive Optical Tomography In The Bayesian Framework},
  pdfauthor={Kit Newton, Qin Li, and Andrew Stuart}
}
\fi

\usepackage[margin=1in]{geometry}
\newcommand{\ssigma}{\alpha}
\newcommand{\eps}{\epsilon}
\newcommand{\lin}{\textrm{lin}}
\newcommand{\NullL}{\textrm{Null}\,\mathcal{L}}
\newcommand{\cov}{\mathcal C}
\newcommand{\post}{\text{post}}
\newcommand{\prior}{\text{prior}}
\newcommand{\rd}{\textrm{d}}
\newcommand{\RTE}{\textrm{RTE}}
\newcommand{\DE}{\textrm{DE}}
\newcommand{\Kn}{\epsilon}

\newcommand{\expDEm}{\mathbb E^{\mmu^{\mathfrak y}_\DE}}
\newcommand{\expRTEm}{\mathbb E^{\mmu^{\mathfrak y}_\RTE}}

\newtheorem{thm}{Theorem}[section]
\newtheorem{prop}[thm]{Proposition}

\newtheorem{lem}[thm]{Lemma}
\newtheorem{cor}[thm]{Corollary}
\theoremstyle{definition}
\newtheorem{assumption}[thm]{Assumption}

\DeclareMathOperator{\KL}{d_\text{KL}}
\DeclareMathOperator{\Hell}{d_\text{Hell}}

\begin{document}
%
%

 \maketitle

\begin{abstract}
Many naturally-occuring models in the sciences are well-approximated
by simplified models, using multiscale techniques. In such settings
it is natural to ask about the relationship between inverse problems
defined by the original problem and by the multiscale approximation.
We develop an approach to this problem and exemplify it in the context
of optical tomographic imaging.

Optical tomographic imaging is a technique for infering the properties of 
biological tissue via measurements of the incoming and outgoing light intensity; it may be used as a
medical imaging methodology. Mathematically, light propagation is modeled 
by the radiative transfer equation (RTE), and optical tomography amounts to 
reconstructing the scattering and the absorption coefficients in the RTE 
from boundary measurements. We study this problem in the Bayesian framework, 
focussing on the strong scattering regime. In this regime the forward 
RTE is close to the diffusion equation (DE). 
We study the RTE in the asymptotic regime where the forward problem
approaches the DE, and prove convergence of the inverse RTE to the inverse DE in both nonlinear 
and linear settings. Convergence is proved by studying the distance between 
the two posterior distributions using the Hellinger metric, and using
Kullback-Leibler divergence.

\end{abstract}

\section{Introduction}
\label{sec:I}

\subsection{Background}

Optical imaging is one of many medical imaging techniques that uses light to probe tissue structure \cite{ren,balreview}. Near infrared light is sent into biological tissue, and the outgoing photon current at the surface of the tissue is then measured. Using these measurements, it is possible to infer properties of the tissue. While traditional imaging methods such as X-ray imaging provide good reconstructions of the tissue's properties, they are typically more expensive and more invasive than optical imaging. Optical imaging can be used for brain, breast, and joint imaging, as well as monitoring blood oxygenation \cite{hielscher1,hielscher2,hielscher_tissue}.

To study optical imaging mathematically, one may use the radiative transfer equation (RTE). The forward RTE describes the dynamics of photons in materials with 
given optical properties. We denote the distribution of particles at location $x$ with velocity $v$ by $f(x,v)$, where $x \in \Omega \subset \mathbb R^d$, $d=2,3$, and $v \in \mathbb S^{d-1}$, the unit sphere in $\mathbb R^d$. The model
enforces particle motion with constant unit speed, and the velocity affects the problem only through the direction of travel of the particle. 
 The optical properties are characterized by two parameters, the scattering coefficient and the absorption coefficient.
The scattering coefficient, denoted $k(x,v,v')$, determined by the probability of 
a photon, currently moving in direction $v$ at position $x$, scattering off a particle 
in the material and changing direction to $v'$. The total absorption coefficient, denoted by $\ssigma(x,v)$, accounts for photon absorbtion into the material 
where they are lost. With this notation established, the RTE is, in its most general form, 
\begin{equation}
\label{gen_rte}
v \cdot \nabla f(x,v) = \int_{\mathbb S^{d-1}} k(x,v,v') f(x,v') \rd v' - \ssigma(x,v) f(x,v).
\end{equation}
Here the gradient operation is with respect to $x$,
as are related contractions of the gradient to a divergence. 
The forward problem is to determine the particle distribution function $f$,
given the optical properties of the medium as characterized by $k$ and $\ssigma.$
Optical imaging amounts to solving a related inverse problem: the map from incoming data (light intensity injected into the tissue) to the measured outgoing data (light intensity collected outside the tissue) is termed the albedo operator, and the absorption and scattering coefficients in RTE are reconstructed from 
knowledge of the albedo operator. There are a number of theoretical 
results concerning the inverse RTE, primarily focussed on the setting
in which the entire albedo map is known: it was shown in~\cite{choulliandstef2}
that the medium is uniquely recoverable in dimension $d=3$, and then 
in~\cite{wang} that the reconstruction is Lipschitz stable. For further
literature surveys see the reviews in~\cite{baljollivet,balmonard}.

Another model for photon dynamics is the diffusion equation (DE). The diffusion
equation typically governs lower energy photons than the RTE, leading to a 
larger scattering effect and less absorption. Let $\rho(x)$ denote the light 
intensity at location $x$ where, as before, $x\in\Omega\subset \mathbb R^d$, 
and let $a(x)$ denote the diffusion coefficient. Then the DE is
\begin{equation*}
-\nabla \cdot\left( a(x)\nabla \rho(x) \right)  =0.
\end{equation*}
In this setup, the map from the Dirichlet data (light intensity injected into the tissue) to the Neumann data (light propagating out) is termed the Dirichlet-to-Neumann (DtN) map, and is used to reconstruct the diffusion coefficient $a(x)$.
Using the DtN map to reconstruct the medium for the elliptic equation is the 
famous Calder\'on problem, widely studied from a
theoretical perspective. Two foundational papers are~\cite{uhlmann}, where
the uniqueness was shown, and~\cite{alessandrini} in which logarithmic 
stability of the inversion was demonstrated. The review \cite{borcea2002electrical} contains
further citations to literature in this area.

It is natural to examine the relation between the two forward models, and to 
understand, from both a physical as well as a mathematial perspective, why 
they give distinct stability performances in the inverse problem. When
$a(x)$ and $(k(x,v,v'),\ssigma(x,v))$ satisfy certain relations, the two 
forward models are asymptotically ``close" when the laser beam is composed 
of low-energy photons. 
In the forward setting, physically, high-energy photons experience little 
scattering before leaving the domain, while in comparison, low-energy photons are scattered frequently by the tissue before being released and measured at the boundary. As a consequence, high-energy photons present a crisp resolution, and the images from low-energy photons are rather blurred. Mathematically, the RTE is taken as the correct forward model, and we can use the Knudsen number to present the number of times that an average photon scatters. In the low-energy regime, the number of times the photon scatters increases, the Knudsen number shrinks to zero, and the RTE converges to the DE in the forward setting. Correspondingly, the inverse RTE is expected to converge to the inverse DE: meaning the information carried in the albedo operator is almost the same as that in the DtN map, and the reconstruction should also converge. This has been numerically observed in~\cite{arridge,hielscher_diffusion,main}, and proved rigorously in~\cite{LiLaiUhlmann}.

The literature refered to thus far focusses on settings in which the
entire albedo or DtN map is known, and this leads to a deep mathematical
theory. However it is arguably far from the practical setting in which
partial and noisy information about these mappings is all that is
available.  The Bayesian formulation of the inverse problem is
useful in this setting as it allows for incorporation of
prior information, partial obeservation and noise level in
a natural fashion. The practicality of this approach was
demonstrated in the monograph \cite{kaipio2006statistical},
and recent work has led to a mathematical framework
\cite{lasanen2012non,lasanen2012nonb,stuart} suitable for
well-posedness analyses \cite{hellinger} and  computations
which blend state-of-the-art computational PDE and computatioal statistical
approaches \cite{ghattas,cotter2013mcmc}. In the Bayesian approach to the solution of the 
inverse problem, all quantities are viewed as random variables, and
the solution is the probability distribution of the unknown quantity 
conditioned on the data \cite{hellinger}. Bayes' theorem allows 
determination of this conditional distribution (the posterior)
from the prior distribution on the unknown and from the likelihood, 
the distribution of the data conditioned on fixing the unknown. 
Our work is focussed on understanding the relationship between the
two inverse problems in this Bayesian setting.

\subsection{Our Contribution}

The goal of this paper is to connect the two inverse problems in optical 
imaging, and specifically to prove convergence of the inverse RTE to the 
inverse DE in the Bayesian framework. Multiscale techniques provide
the desired estimates on the forward problem and we show how these
may be transfered to the Bayesian inverse problems. To this end we make the
following contributions: 

\begin{itemize}

\item we provide multiscale-based error estimates which relate
the solution of the forward
problems for the RTE and DE -- see Theorem \ref{thm:convergence} for which
we provide a formal asymptotic justification in the main body of the paper,
and a rigorous proof in the appendix; 

\item we compare the two posterior distributions for the RTE and DE, 
measuring distance between them in the Kullback-Leibler divergence and the Hellinger distance in the optically thick regime (zero limit of the Knudsen number)
when the scattering coefficient is large -- see Theorems \ref{t:KL1}
and \ref{thm:hellinger_nonlinear};

\item we extend the convergence result linking posterior distributions
to a setting in which the albedo operator's dependence on the medium 
is approximated by linearization -- see Theorem \ref{t:H2} and 
Corollary \ref{c:H2}.

\end{itemize}

The approach we adopt will apply to other Bayesian inverse problems
whose forward models are linked through multiscale analyses.
The paper is organized as follows. In section \ref{sec:S} we provide 
the mathematical setting for our work, including discussion of the
Bayesian formulation of inverse problems, and the diffusion limit 
of the radiative transfer equation. In section \ref{sec:N} 
we estimate the distance between the Bayesian solution
of the RTE and DE inverse problems, and in section \ref{sec:L}
we address the same question in the linearized setting. 
We conclude in section \ref{sec:C}.

\section{The Setting}
\label{sec:S}

In this section we establish the mathematical framework within which
all our results are derived. In subsection \ref{bayes} we describe
Bayesian inverse problems in general. We then discuss the setting of
linear inverse problems with Gaussian priors and Gaussian additive noise,
in which the posterior is also Gaussian; and we discuss linearization
of the forward operator to obtain an approximate
Gaussian posterior. Subsection \ref{asymp} 
describes the forward problems from the RTE and for the DE, providing
error estimates linking their solutions in the small Knudsen number
regime. In subsection \ref{invintro} we formulate the Bayesian
inverse problem for the RTE and DE.
Subsection \ref{lin} is denoted to linearization of the forward
operator for RTE and DE, and hence forms the basis for defining an
approximate Gaussian posterior distributions for their respective
inverse problems.

\subsection{Bayesian Inversion} \label{bayes}

Consider the inverse problem of finding $\sigma$ from $y$ where
\begin{equation}
y = \mathcal G(\sigma) + \eta\,, \label{abstractinvproblem}
\end{equation}
$\mathcal G$ is a known forward map that takes parameter to the 
data space, and $\eta$ a noise pollution. In the 
Bayesian formulation of inversion $y,\sigma$ and $\eta$ are viewed
as random variables, linked by equation \eqref{abstractinvproblem}, 
and assumption that the prior distribution on $\sigma$, 
denoted by $\mu_0$, and the distribution of the noise $\eta$,
denoted by $\mu_\text{error}$, are known. The objective
is to find the conditional distribution on $\sigma$ given $y$
(denoted $\sigma|y$.) In this paper we will assume that
$\eta$ is independent of $\sigma$ {\em a priori} and denote
the distribution of $y$ given $\sigma$, which is then a translation
by $\mathcal{G}(\sigma)$ of $\mu_\text{error}(\eta)$, by $\mu^{\sigma}(y).$
We will concentrate on the commonly occuring case in which $\eta$ is
in a function space and the data $y$ is finite dimensional; then
$\mu_\text{error}(\eta)$ may be identified with its Lebsegue density,
whilst $\mu_0$ and $\mu^y$ are measures on a separable Banach space.

If we denote by $\mu^y$ the posterior distribution on $\sigma$ given observation
$y$ then Bayes' theorem gives 
\begin{equation}\label{bayestheorem}
\mu^y(d\sigma) = \tfrac{1}{Z}\mu^\sigma(y) \mu_0(d\sigma),
\end{equation}
where
\[
Z= \int_{X} \mu^\sigma(y)\rd\mu_0(d\sigma)
\]
and $X$ is a subset of a separable Banach space which contains the
support of $\mu_0$; then $Z$ normalizes $\mu^y$ to a  probability density.
If we make the additional assumption that both $\mu_0$
and $\mu_\text{error}(\eta)$ are Gaussian and finite dimensional
so that $\mu_0=\mathcal N(m_0,\cov_\prior)$
and $\mu_\text{error}(\eta)=N(0,\mathcal C_\text{prior})$ then 
we may write a formula for the Lebesgue density of the posterior:
\begin{equation}\label{bayestheorem2}
\mu^y(\sigma) = 
\tfrac{1}{Z}\exp \left( - \left( \sigma - m_0 \right)^\top \mathcal C_\text{prior}^{-1} \left( \sigma-m_0\right) - \left( y- \mathcal G(\sigma) \right)^\top\mathcal C_\text{error}^{-1} \left( y- \mathcal G(\sigma)\right)\right)\,.
\end{equation}
We note that analogous formulae are also available in the infinite dimensional case; see \cite{stuart, hellinger} and the references therein.

In optical tomography, one has two fundamental models for describing light 
propagation: the radiative transfer equation (RTE), and the diffusion 
equation (DE). We will denote the solution of the respective inverse
problems by $\mu^{y}_\DE(\sigma)$ and $\mu^{y}_\RTE(\sigma)$. 
This paper is primarily concerned with showing that these two distributions
are close in the small Knudsen number regime, and quantifying the difference. 
There are multiple ways to quantify the distance between two probability distributions $\mu$ and $\mu'$. We will use the
Kullback-Leibler (KL) divergence and the Hellinger distance. 
If $\mu$ has density with respect to $\mu'$ and $\mu$ has support in $X$
defined as above, then the KL divergence is given by
\begin{equation*}
\KL(\mu,\mu') = \int_{X}  \log \frac{d\mu}{d\mu'}(\sigma) \rd \mu(d\sigma);
\end{equation*}
if $\mu$ and $\mu'$ have density with respect to common reference
measure $\lambda$, with support in $X$ 
defined as above, then the Hellinger distance is given by
\begin{equation*}
d_\text{Hell}(\mu,\mu')^2 = \frac{1}{2} \int_{X} \left( \sqrt{\frac{d\mu}{d\lambda}} - \sqrt{\frac{d\mu'}{d\lambda}}\right)^2 \rd \lambda.
\end{equation*}
These formulae have interpretations in the infinite dimensional setting; see the appendix of \cite{hellinger}. The KL divergence has an information theoretic interpretation which makes it attractive. However, the Hellinger metric is particularly useful because, for square integrable test functions, it translates directly into bounds of differences of expectations of test functions, see Lemma 7.14 in \cite{hellinger}. The square root of the KL divergence bounds the Hellinger metric, but often sharper bounds on differences of expectations of test functions are obtained by studying the Hellinger distance directly.
KL divergence was used to quantify the error incurred when approximating
posterior distributions in \cite{marzouk2009stochastic} in finite dimensions,
and the Hellinger metric was used in \cite{cotter2010approximation} in
the infinite dimensional setting required in this paper.

In some contexts the unknown $\sigma$ is naturally a positive random variable and so
we seek instead $u$ where $\sigma=\exp(u)$. If we define $\cF=\cG \circ \exp(\cdot)$
then the inverse problem \eqref{abstractinvproblem} becomes 
\begin{equation}
\label{abstractinvproblem_2}
y=\cF(u)+\eta.
\end{equation}
Often we have an approximate solution $u_0$ to \eqref{abstractinvproblem_2}
and it is natural to seek a solution which deviates slightly from this. In this
situation we write $u=u_0+v$ and linearize \eqref{abstractinvproblem_2} to obtain
$$y\approx \cF(u_0)+D\cF(u_0)v+\eta.$$
This suggests studying the linear inverse problem
\begin{equation}
\label{abstractinvproblem_lin}
z=Gv+\eta
\end{equation}
where $z=y-\cF(u_0)$ and $G=D\cF(u_0).$
If we put  Gaussian prior $\mathcal N(m_\prior,\cov_\prior)$ on $v$ then the posterior 
on $v|z$ is also Gaussian $\mathcal N(m_\post,\cov_\post)$
determined by 
\begin{equation}
\mathcal C_\post^{-1} = \mathcal C_\prior^{-1} + G^T \mathcal C_\text{error}^{-1} G\,,\quad\text{and}\quad  m_\post = m_\prior + \mathcal C_\post G^T \mathcal C_\text{error}^{-1} \left(z -G m_\prior \right).
 \label{stuart620}
\end{equation}
These formulae can also be interpreted in the infinite dimensional setting; 
see \cite{lehtinen1989linear} and further citations 
in \cite{stuart,hellinger,stuart2015}.

When Bayesian inversion is based on a nonlinear forward model, 
characterization of the resulting non-Gaussian
posterior distribution can be quite complicated, requiring MCMC or SMC 
techniques \cite{taeb}. One possible approach to
deal with this is to perform the linearization described above and work
with Gaussian priors and posterior distributions, leading to closed
form solutions. These can be augmented with constraints by means of rejection
sampling based on independent sampling from the Gaussian posterior.

\subsection{Diffusion Limit Of The RTE}\label{asymp}
We consider the RTE \eqref{gen_rte} in the setting where the
the absorption coefficient $\alpha$ and the scattering coefficient $k(x,v,v')$ are set to
$$\alpha(x,v) = k(x,v,v')=\Kn^{-1}\sigma(x),$$
where $\Kn$ is the Knudsen number. 
The thickness of the material physically corresponds to the number of times a photon scatters between being injected in a medium and escaping. The physical quantity is termed the Knudsen number, which stands for the ratio of mean free path and the domain length. The mean free path is the average distance a particle travels before being scattered. When the Knudsen number is small, photons, on average, scatter many times before they are emitted, and the material is thus regarded as optically thick.

In equation \eqref{gen_rte} $\rd v$ denotes normalized unit measure, meaning
\begin{equation*}
\braket{1}_v=\int_{\mathbb S^{d-1}} 1\rd v = 1\,,
\end{equation*}
where we have used the notation $\braket{\cdot}_v$ to denote normalized integration over $v$.
Thus equation \eqref{gen_rte} may be written as 
\begin{equation*}
v \cdot \nabla f = \frac{1}{\Kn}\sigma \mathcal L f, \label{newrte}
\end{equation*}
where the collision operator is
\begin{equation}\label{eqn: bam}
\mathcal L f = \int_{\mathbb S^{d-1}} f(x,v') \rd v' - f = \braket{f}_v -f\,.
\end{equation}

To ensure a unique solution we impose an incoming boundary condition, 
the analogue of a Dirichlet boundary condition for equations lacking 
velocity dependence. To this end define 
\begin{equation*}
\Gamma_\pm =  \{ (x,v): x \in \partial \Omega, \pm v \cdot n_x>0 \}\,
\end{equation*}
which denotes the collection of coordinates on the boundary 
$x\in\partial\Omega$ on which the velocity $v$ points in/out of the domain,
respectively where $\pm v\cdot n_x>0$. Here $n_x$ is the normal vector at point $x$ pointing out of $\Omega$. The incoming boundary condition is imposed on $\Gamma_-$. We also define, for any fixed $ y\in\partial\Omega$,
\begin{equation*}
\Gamma_\pm(y) = \{ (x,v): x=y, \pm v \cdot n_y>0 \}\,,
\end{equation*}
and set
\begin{equation*}
\Gamma = \Gamma_+\cup\Gamma_-\,,\quad\text{and}\quad \Gamma(y) = \Gamma_+(y)\cup\Gamma_-(y)\,.
\end{equation*}
For a unique solution to \eqref{gen_rte}, boundary conditions must be imposed on $\Gamma_-$ as follows:
\begin{equation*}
f|_{\Gamma_-} = \phi(x,v)\,.
\end{equation*}

Combining the foregoing considerations we obtain
\begin{equation}\label{rteknud}
\begin{cases}
v \cdot \nabla f = \frac{1}{\Kn} \sigma \mathcal L f\,,\quad (x,v)\in\Omega\times\mathbb{S}^{d-1}\\
f|_{\Gamma_-} = \phi(x,v)
\end{cases}\,,
\end{equation}
with $\mathcal{L}$ as defined in \eqref{eqn: bam}. The domain $\Omega$ has a smooth $C^1$ boundary $\partial\Omega$. In the small $\Kn$ regime, it was conjectured in~\cite{diff_limit} and then proved in~\cite{diff_limit_proof} that the equation is asymptotically close to the 
following diffusion equation: 
\begin{equation}\label{eqn:diff_limit}
\begin{cases}
-\nabla \cdot \left( \frac{1}{\sigma} \nabla \rho \right)  = 0, \quad x\in\Omega\subset\mathbb{R}^d\\
\rho \big|_{\partial \Omega} = \xi(x)
\end{cases}\,,
\end{equation}

We make this convergence explicit under the following assumptions.
\begin{assumption}\label{ass:sigma}
The functions $\sigma, \phi$ and $\xi$ characterizing the medium and 
the boundary conditions are smooth functions, bounded in the following
sense:
\begin{itemize}
\item the admissible medium is bounded, meaning there is a constant $C_1$ so that:
\begin{equation*}
\max\{\|\sigma\|_{L_\infty(\Omega)}\,,\|\sigma^{-1}\|_{L_\infty(\Omega)}\,,\|\nabla\left(\sigma^{-1}\right)\|_{L_\infty(\Omega)}\}<C_1\,;
\end{equation*}
\item and the boundary conditions are smooth and bounded, meaning:
\begin{equation*}
\max\{\|\xi\|_{L_\infty(\partial\Omega)}\,,\|\phi\|_{L_\infty(\Gamma)}\}<C_1\,.
\end{equation*}
\end{itemize}
We also term the set of admissible media:
\begin{equation}\label{eqn:admissible}
\mathcal{A} = \{\sigma\in C^{3}(\Omega): \max\{\|\sigma\|_{L_\infty(\Omega)}\,,\|\sigma^{-1}\|_{L_\infty(\Omega)}\,,\|\nabla\left(\sigma^{-1}\right)\|_{L_\infty(\Omega)}\}<C_1\}\,.
\end{equation}
Here $C^{3}$ is the collection of third-order differentiable function set.
\end{assumption}

With this assumption, we first have the uniform boundedness of the Neumann data over $\mathcal{A}$.

\begin{proposition}[\cite{gilbarg1983}]\label{prop:uniform_bound_Neumann}
Suppose $\rho$ solves~\eqref{eqn:diff_limit} with medium $\sigma$ and the smooth boundary condition $\xi$ satisfying Assumption~\ref{ass:sigma}, then there is a constant $C$ that only depends on $C_1$ and $\Omega$, so that
\begin{equation}
\sup_{\sigma\in\mathcal{A}}\|\frac{1}{\sigma}\partial_n\rho\|_\infty <C\,.
\end{equation}
\end{proposition}

Note that we assume only that the medium is smooth enough so the Neumann data is bounded. The regularity of the medium could certainly be relaxed but we do not pursue that direction in this paper. The key point here is to have the uniform bound over the set $\mathcal{A}$.

\begin{thm}\label{thm:convergence} Suppose $f(x,v)$ satisfies equation~\eqref{rteknud} with smooth boundary conditions and $\rho(x)$ 
solves \eqref{eqn:diff_limit}. Then, as $\Kn\to 0$, $f(x,v) \to \rho(x),$
assuming suitable compatibility relationships between the boundary data 
$\phi$ and $\xi$ of the two equations.
In particular, with compatible boundary conditions at different orders, one approximates $f$ through different forms:
\begin{itemize}
\item if $\phi(x,v) = \xi(x)$:
\[
\|f-\rho\|_{L_\infty(\Omega\times\mathbb{S}^{d-1})} < C_\mathcal{A}\Kn\,;
\]
\item if $\phi(x,v) = \xi(x) -\Kn \frac{1}{\sigma(x)}v(x)\cdot\nabla\xi(x)$:
\[
\|f-\rho +\frac{\Kn}{\sigma}v\cdot \nabla\rho\|_{L_\infty(\Omega\times\mathbb{S}^{d-1})} < C_\mathcal{A}\Kn^2\,.
\]
\end{itemize}
Here the constant $C_\mathcal{A}$ depends on $C_1$, the upper bound in Assumption~\ref{ass:sigma} for the admissible set.
\end{thm}
We leave the rigorous proof to the appendix, and present here the formal 
perturbation expansion derivation; the latter is useful in building intuition.
\begin{proof}[Sketch Proof: Perturbation Expansion]
We will use the standard asymptotic expansion technique in $\Kn$ away from the boundary. Set
\begin{equation*}
f_\text{in} = f_0 + \text{$\Kn$} f_1 + \text{$\Kn$}^2 f_2 + \cdots\,.
\end{equation*}
Plugging the expansion into \eqref{rteknud}, we obtain
\begin{equation*}
v \cdot \nabla f_0 + \text{$\Kn$} v \cdot \nabla f_1 + \text{$\Kn$}^2 v \cdot \nabla f_2 + \cdots  ~=~ \frac{1}{\text{$\Kn$}} \sigma \mathcal L [f_0 + \text{$\Kn$} f_1 + \text{$\Kn$}^2 f_2 + \cdots ].
\end{equation*}
Multiplying by $\Kn$ and equating in powers of $\Kn$ gives
\begin{align*}
\text{$\Kn$}^0:& \quad \mathcal L [f_0] = 0 \,,\\
\text{$\Kn$}^1:& \quad v \cdot \nabla f_0 =\sigma \mathcal L[f_1] \,,\\
\text{$\Kn$}^2:& \quad v \cdot \nabla f_1 = \sigma \mathcal L [f_2] \,.
\end{align*}
The zeroth order expansion indicates that $f_0$ is in the null space of $\mathcal{L}$. From equation~\eqref{eqn: bam} we deduce that $f_0$ must be velocity 
independent, and thus we write $f_0(x,v)=\rho(x)$. 
With this expression, considering the equation at $\mathcal{O}(\Kn)$, we have
\begin{equation*}
f_1 = \mathcal L^{-1} \left[ \frac{1}{\sigma} v \cdot \nabla \rho \right]\quad\Rightarrow \quad f_1 = - \frac{1}{\sigma} v \cdot \nabla \rho\,.
\end{equation*}
Here we have used the fact that $\mathcal L$ is one-to-one on the domain 
($\NullL)^\perp$ and that $v \cdot \nabla \rho\perp(\NullL)^\perp$, since $v$
integrates to zero on the unit sphere. To close the system we consider 
the equation at $\mathcal{O}(\Kn^2)$, substituting $f_0=\rho$ and $f_1 = - \frac{1}{\sigma} v \cdot \nabla \rho$ to obtain
\begin{equation*}
-v \cdot \nabla \left(\frac{1}{\sigma} v \cdot \nabla \rho\right) = \sigma \mathcal L [f_2].
\end{equation*}
Integrating the equation on both sides with respect to $v$ and taking into account the fact that $\int_{\mathbb S^{d-1}} \mathcal L[f_2] \rd v=0$, we have, using the summation
convention,
\begin{align*}
0 &  = - \int_{\mathbb S^{d-1}} v \cdot \nabla \left( \frac{1}{\sigma} v \cdot \nabla \rho \right) \rd v = - \int_{\mathbb S^{d-1}}  v_i v_j \partial_i \left( \frac{1}{\sigma} \partial_j \rho\right) \rd v  = -  C_d \partial_i \left( \frac{1}{\sigma}\partial_i \rho\right) \\
& = - C_d \nabla \cdot\left( \frac{1}{\sigma} \nabla \rho \right).
\end{align*}
implying that
\begin{equation*}
- \nabla \cdot\left( \frac{1}{\sigma} \nabla \rho \right)=0\,.
\end{equation*}
Here we have used that 
\begin{equation}
\label{eq:needthis}
\int_{\mathbb S^{d-1}} v_iv_j\rd{v}=C_d\delta_{ij}, \quad C_d:=\int_{\mathbb S^{d-1}} v_i^2\rd{v}.
\end{equation}
Note that $C_d$ depends on dimension.
For example, in $\mathbb S^2$, $C_d=1/3$. Thus, we have shown that the radiative transfer equation in the diffusion limit becomes the diffusion equation, which concludes the sketch proof of the theorem.
\end{proof}

\begin{remark}
In the preceding formal derivation we have ignored boundary conditions.
In practice, unless these are chosen carefully, there will be a mismatch 
between the DE and the small $\Kn$ solution of the RTE near the boundary.
The boundary conditions stated in the theorem give different levels
of consistency between the two equations, and hence lead to differing
error estimates. See the proof in the appendix for details.
When the boundary conditions are incompatible the analysis is considerably
more subtle -- see \cite{diff_limit,layer_correction,layer_correction_l2} 
for details.
\end{remark}
%

\subsection{Inverse Problems for the RTE and DE} \label{invintro}
We now define Bayesian inverse problems for the RTE and DE,
relating to partial and noisy observations of the albedo and DtN
operators respectively. The first ingredient is definition of
the forward map $\mathcal{G}$, which we now do for the RTE and DE
equations. We conclude this subsection with a discussion of the
prior distribution, which we choose in common between the RTE and DE
settings.

In optical tomography, high energy light with a known intensity is injected 
into the material, and detectors are placed on the tissue boundary to 
collect the light current emitted. For the RTE, the albedo operator is
defined by $\mathcal H^\RTE$ which is a $\sigma-$dependent linear
transformation of boundary data $\phi$ into the measurement space,
defined by
\begin{equation}\label{eqn:h_RTE}
\mathcal H^\RTE(\sigma)\phi=h^\RTE\,.
\end{equation}
where
\begin{equation}\label{eqn:calH_RTE}
h^\RTE(x) = -\frac{1}{C_d \eps} \int_{\Gamma(x)} v \cdot n f(x,v)\rd v\,,
\end{equation}
and $f$ satisfies~\eqref{rteknud}. It is important to note that, whilst
$\mathcal H^\RTE$ is linear in its action on $\phi$, it depends nonlinearly
on the unknown medium $\sigma.$ The inverse problem of reconstructing
$\sigma$ from measurements of $h^\RTE$ is thus a nonlinear inverse problem.
 
In practice, finitely many smooth incoming data $\phi_k$ are injected and 
finitely many measurements are made at the boundary for each $\phi_k$;
we assume that these measurements may be expressed as linear functionals
$l_j$ of $h^\RTE$. We thus define the forward map to be inverted by
\begin{equation}
\label{eq:GRTE}
\mathcal{G}^\RTE(\sigma)_{j,k}=l_j(\mathcal H^\RTE(\sigma)\phi_k),
\end{equation}
where $(j,k)\in\{1,\cdots,J\}\otimes\{1,\cdots,K\}.$ 
We assume additive Gaussian noise $\eta$ to obtain the compact
representation of the inverse problem
\begin{equation}\label{eqn:calG_RTE}
\mathfrak y = \mathcal{G}^\RTE(\sigma) + \eta\,,
\end{equation}
where $\eta \in \mathbb{R}^{JK}$ is drawn from a Gaussian distribution
which we assume to have the form 
\begin{equation}\label{eqn:eta_pollution}
\eta\sim \mathcal{N}(0,\gamma^2\mathbb{I})\,,
\end{equation}
meaning
\begin{equation}\label{eqn:rte_likelihood}
\mathfrak y\,\big|\,\sigma\sim \mathcal{N}(\mathcal{G}^\RTE(\sigma),\gamma^2\mathbb{I})\,.
\end{equation}

For the DE model the situation is analogous.
The map that takes the Dirichlet data to the Neumann outflow is termed the DtN map and is defined by
\begin{equation}\label{eqn:h_DE}
\mathcal H^\DE(\sigma)\phi=h^\DE\,,
\end{equation}
where
\begin{equation}\label{eqn:calH_DE}
h^\DE(x) = \frac{1}{\sigma} \frac{\partial \rho}{\partial n}(x)\,
\end{equation}
and $\rho$ satisfies~\eqref{eqn:diff_limit}. 
In practice, finitely many incoming data $\xi_k$ are injected and finitely 
many linear functionals $l_j$ of $h^\DE$ are measured, noisily, leading
to an inverse problem of the form
\begin{equation}\label{eqn:calG_DE}
\mathfrak y = \mathcal{G}^\DE(\sigma) + \eta\,,
\end{equation}
where $\eta \in \mathbb{R}^{JK}$ denotes observational noise and
where the forward map is defined by  
\begin{equation}
\label{eq:GDE}
\mathcal{G}^\DE(\sigma)_{j,k}=l_j(\mathcal H^\DE(\sigma)\phi_k),
\end{equation}
where $(j,k)\in\{1,\cdots,J\}\otimes\{1,\cdots,K\}.$
For simplicity we assume the same noise model \eqref{eqn:eta_pollution}
for $\eta.$

Together with \eqref{eqn:eta_pollution} and the assumption that $\eta$
and $\sigma$ are {\em a priori} independent,  \eqref{eqn:calG_RTE} 
and \eqref{eqn:calG_DE} define the likelihood for a Bayesian formulation
of the inverse problem of determining $\sigma$ from $\mathfrak y$ from RTE
and DE respectively.
We now define the prior on $\sigma$, which we will choose in common
between the two inverse problems. To this end recall the set
\eqref{eqn:admissible} and define prior distribution $\mu_0(d\sigma)$
to be a probability measure supported on $\mathcal{A}:$

\begin{assumption}
\label{a:add}
The prior measure $\mu_0$ is supported on an infinite dimensional separable
Banach space, and the support is contained in the admissible set $\mathcal{A}$
given by \eqref{eqn:admissible}:
\begin{equation*}
\int_\mathcal{A}1\rd\mu_0(\sigma) = 1\,.
\end{equation*}
\end{assumption}

\noindent In our case $\mathcal{A}$ is a subset of $C^3$. If we further relax the regularity assumptions, to let $\mathcal{A}$ be a subset of $W^{1,\infty}$, for example, then $W^{1,\infty}$ is not separable. But it is possible to construct
useful measures with support in $W^{1,\infty}$ which are separable, for example
through the closure of sets of random series expansions; see
\cite{hellinger} for a related  example in $L^{\infty}.$

Bayes' theorem~\eqref{bayestheorem} for both models is then given by
\begin{equation}
\label{eq:bayes1}
\mu^{\mathfrak y}_\RTE(d\sigma) = \frac{1}{Z^\RTE}\mu^\sigma_\RTE(\mathfrak y)\mu_0(d\sigma)\,, \quad\text{and}\quad \mu^{\mathfrak y}_\DE(\sigma) = \frac{1}{Z^\DE}
\mu^\sigma_\DE(\mathfrak y)\mu_0(d\sigma)\,,
\end{equation}
where
\begin{equation}
\label{eq:bayes2}
\mu^\sigma_\RTE(\mathfrak y) = \exp\left(-\frac{1}{2\gamma^{2}}\|\mathfrak y - \mathcal{G}^\RTE(\sigma)\|^2_2\right)\,,\quad\text{and}\quad \mu^\sigma_\DE(\mathfrak y) = \exp\left(-\frac{1}{2\gamma^{2}}\|\mathfrak y - \mathcal{G}^\DE(\sigma)\|^2_2\right)\,.
\end{equation}
The functions $\mathcal{G}^\RTE$ and $\mathcal{G}^\DE$ are here both viewed
as mappings from $\mathcal{A}$ into $\mathbb{R}^{JK}.$
The normalization factors are given by 
\begin{equation}
\label{eq:bayes3}
Z^\RTE = \int_{\mathcal{A}} \mu^\sigma_\RTE(\mathfrak y)\rd\mu_0(\sigma)\,,\quad\text{and}\quad Z^\DE = \int_{\mathcal{A}} \mu^\sigma_\DE(\mathfrak y)\rd\mu_0(\sigma)\,.
\end{equation}
Note also that the likelihoods $\mu^\sigma_\RTE(\mathfrak y)$ and $\mu^\sigma_\DE(\mathfrak y)$ are, for fixed $\sigma$, proportional to
densities on $\mathbb{R}^{JK}$; hence we write them as functions of $y$.
On the other hand $\mu^{\mathfrak y}_\RTE(d\sigma), \mu^{\mathfrak y}_\DE(d\sigma)$
and $\mu_0(d\sigma)$ are measures with support in $\mathcal{A}$,
a subset of an infinite dimensional separable Banach space.

Theorem \ref{thm:convergence} shows that given compatible $\phi$ and $\xi$, $h^\RTE$ and $h^\DE$ are close when the Knudsen number $\Kn$ is small, 
so that $\mathcal{G}^\RTE$ and $\mathcal{G}^\DE$ are close for every fixed 
$\sigma$, when $\Kn$ is small. 
In section \ref{sec:N} we use these facts to demonstrate the 
convergence of  $\mu^\sigma_{\RTE}$ to $\mu^\sigma_{\DE}$ as $\Kn \to 0.$

\subsection{Linearized Albedo Operator And DtN Map} \label{lin}
We derive linearized versions of the albedo operator and the DtN map
by assuming that the unknown medium $\sigma$ is close to a known background 
$\sigma_0$. In order to  enforce positivity, we assume $\sigma = e^{u}$, 
define $u_0$ so that $\sigma_0 = e^{u_0}$, and find equations satisfied 
by the perturbation $w = u-u_0$. 
The corresponding inverse problem amounts to 
reconstructing $w$ using the measurements and some known information 
computed using the background medium $\sigma_0$. We express
the admissible set $\mathcal{A}$ given in \eqref{eqn:admissible} 
on the log-scale and write
\begin{equation}\label{eqn:admissible_u}
\mathcal{A}_u = \{u\in C^{3}(\Omega): \max\{\|e^u\|_{L_\infty(\Omega)}\,,\|{e^{-u}}\|_{L_\infty(\Omega)}\,,\|\nabla\left(e^{-u}\right)\|_{L_\infty(\Omega)}\}<C_1\}\,.
\end{equation}
This implies
\begin{equation}\label{eqn:bound_u}
\sup_{u \in \mathcal{A}_u} \|u\|_{L_\infty(\Omega)}<C_2 := \log{C_1}\,.
\end{equation}

To start, we recall equation~\eqref{rteknud},
\begin{equation}\label{eqn:rte_eu}
\begin{cases}
v \cdot \nabla f = \tfrac{1}{\eps} e^u \mathcal L f \\
f \big|_{\Gamma_-} = \phi(x,v)
\end{cases}\,.
\end{equation}
We assume that there is a background scattering coefficient $u_0\in\mathcal{A}_u$, and that $w(x)\in C^3(\Omega)$ is a small fluctuation of $u$ around
the background $u_0$: $w=u-u_0$. Then  
\begin{equation}\label{eqn:assumption_sigma}
\|w(x) \|_{L_\infty(\Omega)}= \|u(x) - u_0(x)\|_{L_\infty(\Omega)} \ll \| u\|_{L_\infty(\Omega)}\,.
\end{equation}
We define a new function $f_\lin$ which solves the RTE with the background scattering coefficient and the same boundary condition,
\begin{equation}
\begin{cases}
v \cdot \nabla f_\lin = \frac{1}{\Kn} e^{u_0}\mathcal L f_\lin \\
f_\lin \big|_{\Gamma_-} = \phi(x,v) \label{rtelin}
\end{cases}\,.
\end{equation}
The difference between $f$ and $f_\lin$, termed $\mathfrak f= f - f_\lin$, then 
satisfies, neglecting terms of $\mathcal{O}(w^2)$, the following 
error equation:
\begin{equation}
\begin{cases}
v \cdot\nabla \mathfrak f =  \frac{1}{\Kn} e^{u_0} \mathcal L \mathfrak f + \frac{1}{\Kn}e^{u_0}w \mathcal L f_\lin\\
\mathfrak f \big|_{\Gamma_-} = 0\label{rtefluc}
\end{cases}\,.
\end{equation}

To extract boundary data from~\eqref{rtefluc}, we define the adjoint equation, with a delta function on the boundary,
\begin{equation}
\begin{cases}
-v \cdot \nabla g = \frac{1}{\Kn} e^{u_0} \mathcal L g \\
g \big|_{\Gamma_+} = \delta_y(x)\label{rteadj} 
\end{cases}\,.
\end{equation}
Here we have used the fact that $\mathcal L$ is self-adjoint, and for the adjoint equation, the incoming boundary condition should be imposed on $\Gamma_+$. We have also imposed a delta function concentrated at $y\in\partial\Omega$. Multiplying \eqref{rtefluc} by $g$ and \eqref{rteadj} by $\mathfrak f$ and subtracting the two obtained equations, integrated over $x$ and $v$, we obtain, upon using Green's identity
\begin{equation*}
\int_{\Gamma_+(y) \cup \Gamma_-(y)} \left( v \cdot n\right) \mathfrak f g \rd x \rd v  = \frac{1}{\Kn} \int_{\Omega}e^{u_0} w\int_{\mathbb{S}^{d-1}} g \mathcal L f_\lin \rd{v}\rd x\,. 
\end{equation*}
Noting that $\mathfrak f|_{\Gamma_-}=0$ and $g|_{\Gamma_+}=\delta_y$, we may further simplify the left hand side, obtaining
\begin{equation}
\int_{\Gamma_+(y)} v \cdot n_y \mathfrak f(y,v) \rd v = \frac{1}{\Kn} \int_\Omega e^{u_0(x)} w(x)\int_{\mathbb{S}^{d-1}} g (x,v) \mathcal L f_\lin(x,v) \rd v \rd x. \label{introtogamma}
\end{equation}


As in the nonlinear case, we have finitely many measurements and experiments conducted. In the $K$ experiments, we use incoming data $\phi_k$, and for each experiment we measure data using the measurement-operator $l_j$:
\begin{equation}\label{eqn:experiments_measurements}
\{\phi_1\,,\cdots,\phi_K\}\,,\quad\{l_1\,,\cdots,l_J\}\,.
\end{equation}
Letting $g_j$ denote the solution to~\eqref{rteadj} with $\delta_{x_j}(x)$ 
on the boundary, and $f_k$, and $f_{\lin,k}$ denote the solutions 
to~\eqref{eqn:rte_eu} and~\eqref{rtelin} with $\phi_k$ as boundary data, 
we define
\begin{equation}\label{eqn:gammaRTE}
\gamma^\RTE_{jk}(x) =- \frac{e^{u_0}}{C_d\Kn^2} \int_{\mathbb S^{d-1}} g_j(x,v) \mathcal L f_{\lin,k}(x,v)\rd v\,,
\end{equation}
and define 
\begin{equation}\label{eqn:fredhom_rte}
G_{jk}^\RTE (w):=\langle\gamma^\RTE_{jk}(x) \,,w\rangle_x\,, 
\end{equation}
with $\langle \cdot, \cdot \rangle_x$ denoting the inner-product defined
by integration over $\Omega$ in the $x$
variable alone.
We note that, by ~\eqref{introtogamma},
\begin{equation}\label{eqn:fredhom_rte2}
G_{jk}^\RTE (w) \approx \mathcal{G}_{jk} ^{\RTE}(\log(u)) - \mathcal{G}_{jk}^{\RTE}(\log(u_0))
\end{equation}
because
\begin{align*}
\mathcal{G}_{jk} ^{\RTE}(\log(u)) - \mathcal{G}_{jk}^{\RTE}(\log(u_0)) & = l_j(\mathcal{H}^\RTE(\log(u))(\phi_k))- l_j(\mathcal{H}^\RTE(\log(u_0))(\phi_k))\\
& =-\frac{1}{C_d\Kn} \int_{\Gamma_+(x_j)} v \cdot n_{x_j} f_k(x_j,v) \rd v+\frac{1}{C_d\Kn}\int_{\Gamma_+(x_j)} v \cdot n_{x_j} f_{\lin,k}(x_j,v) \rd v\\
&\approx -\frac{1}{C_d\Kn} \int_{\Gamma_+(x_j)} v \cdot n_{x_j} {\mathfrak f}_k(x_j,v) \rd v\,\\
&=G_{jk}^\RTE (w).
\end{align*}
Here $\mathfrak f_k= f_k - f_{\lin,k}.$

%

Deriving the linearized DtN map for the diffusion equation is similar. For ease of notation we start with the form of the diffusion equation as in~\eqref{eqn:diff_limit}, where the scattering coefficient is shown in the denominator. 
For positivity, we use $\sigma = e^u$, and $\sigma_0 = e^{u_0}$ as before.
We now derive an equation for $\tau=\rho-\rho_{lin}$, first noting that
\begin{equation*}
\begin{cases}
-\nabla \cdot \left( \frac{1}{e^{u}} \nabla \rho\right) =- \nabla \cdot \left( \frac{1}{e^{u_0}e^{w}} \nabla (\rho_\lin+\tau)\right)  = 0  \\
\rho \big|_{\partial \Omega} = \xi(x)
\end{cases}\,,
\end{equation*}
where $\rho_\lin$ solves:
\begin{equation}\label{eqn:diff_lin}
\begin{cases}
-\nabla \cdot \left( \frac{1}{e^{u_0}} \nabla \rho_\lin\right)  = 0  \\
\rho_\lin \big|_{\partial \Omega} = \xi(x)
\end{cases}\,.
\end{equation}
Subtracting the two equations and neglecting higher order terms, we have
\begin{equation}
\begin{cases}
 -\nabla \cdot \left(\frac{1}{e^{u_0}} \nabla \tau\right)= -\nabla \cdot \left( \frac{w}{e^{u_0}} \nabla \rho_\lin\right)\\
\tau \big|_{\partial \Omega} = 0
\end{cases}\, \label{tildediff}
\end{equation}
with $w=u-u_0$ as in the RTE case.

We also define $\rho_g$ that solves the adjoint equation:
\begin{equation}
\begin{cases}
-\nabla \cdot \left(\frac{1}{e^{u_0}} \nabla \rho_g \right) =0\\
\rho_g \big|_{\partial \Omega} = \delta_y\label{adjdiff}
\end{cases}\,.
\end{equation}
Multiplying~\eqref{tildediff} by $\rho_g$,~\eqref{adjdiff} by $\tau$ and integrating over $\Omega$, we obtain:
\begin{align}\label{eqn:integratedeqn_diff}
\int_\Omega\frac{w}{e^{u_0}}\nabla\rho_\lin\cdot\nabla\rho_g\rd{x} = -\int_{\partial\Omega}\frac{w}{e^{u_0}}\frac{\partial \rho_\lin}{\partial n}\rho_g\rd{x}+ \int_{\partial\Omega}\frac{1}{e^{u_0}}\frac{\partial\tau}{\partial n}\rho_g\rd{x}=-\frac{w(y)}{e^{u_0(y)}}\frac{\partial\rho_\lin(y)}{\partial n_y}+\frac{1}{e^{u_0(y)}}\frac{\partial\tau(y)}{\partial n_y}\,.
\end{align}

Similarly to the nonlinear case, we conduct finitely many experiments and 
make finitely many measurements as in~\eqref{eqn:experiments_measurements}. Defining
\begin{equation}\label{eqn:gammaDE}
\gamma_{jk}^\DE(x)  = \frac{1}{e^{u_0}} \nabla \rho_{\lin,k} \cdot \nabla \rho_{g,j}\,,
\end{equation}
where $\rho_{g,j}$ satisfies~\eqref{adjdiff} with $\delta_{y_j}$ boundary condition and $\rho_{\lin,k}$ satisfies~\eqref{eqn:diff_lin} with $\xi_k$ with as the boundary condition. Using the approximation that $\tau$ satisfies \eqref{tildediff}, we write~\eqref{eqn:integratedeqn_diff} as
\begin{equation}\label{eqn:fredhom_diff}
\langle\gamma^\DE_{jk}(x) \,,w\rangle_x = \mathcal{G}_{jk}^{\DE}(\log(u)) - \mathcal{G}^{\DE}_{jk}\log((u_0)) = G^\DE_{jk}(w)\,,
\end{equation}
where again $G^\DE_{jk}$ is the linearized approximation and we have used the estimate
\begin{align*}
\mathcal G_{jk}^\DE(\log(u))- \mathcal G^\DE_{jk}(\log(u_0)) & = \frac{1}{e^u} \frac{\partial \rho_k}{\partial n}(x_j) -\frac{1}{e^{u_0}}\frac{\partial\rho_{\lin,k}}{\partial n}(x_j)\\
&= \frac{1}{e^{w} e^{u_0}} \frac{\partial(\tau_k + \rho_{\lin,k})}{\partial n}(x
_j)  -\frac{1}{e^{u_0}}\frac{\partial\rho_{\lin,k}}{\partial n}(x_j)\\
& \approx \frac{1}{e^{u_0} }\frac{\partial \tau_k }{\partial n }(x_j) - \frac{w}{e^{u_0} } \frac{\partial \rho_{\lin,k}}{\partial n}(x_j)\,.
\end{align*}
Here $\tau_k=\rho_k-\rho_{lin,k}.$ This defines the (linear) action of $G^{DE}_{jk}$ on $w$.
%
Notice that the linearized albedo operator~\eqref{eqn:fredhom_rte} and the linearized DtN map~\eqref{eqn:fredhom_diff} have the same format: they are both Fredholm first type integrals, determined by the $\gamma_{jk}^\RTE$ and $\gamma_{jk}^\DE$ respectively defined in~\eqref{eqn:gammaRTE} and~\eqref{eqn:gammaDE}. To show the convergence of the two problems amounts to showing, in the small $\Kn$ regime, the convergence of the two forward maps $\gamma_{jk}^\RTE\sim\gamma_{jk}^\DE$ for all $j$ and $k$, and the convergence of the data $G^\RTE_{jk}(w)$ to $G^\DE_{jk}(w)$ for reasonably small $w$.

\section{Nonlinear Inverse Problems}
\label{sec:N}
In this section we analyze the distance between the posterior distributions
of the nonlinear inverse problems defined by the RTE and DE, namely
$\mu^{\mathfrak y}_\RTE(\sigma)$ and $\mu^{\mathfrak y}_\DE(\sigma)$ respectively.
We consider the setting in which the Knudsen number $\Kn$ is small. We show
that the two distributions converge in the Kullback-Leibler divergence 
and the Hellinger distance as $\Kn \to 0$. 
The three subsections concern, in turn, the following convergence results
as  $\Kn \to 0$:
\begin{itemize}
\item[1.] convergence of the forward map $\mathcal{G}^\RTE(\sigma)$ to $\mathcal{G}^\DE(\sigma)$ for a fixed list of $(\phi_k,l_j)$;
\item[2.] convergence of the KL divergence between $\mu^{\mathfrak y}_\RTE(\sigma)$ and $\mu^{\mathfrak y}_\DE(\sigma)$ to zero;
\item[3.] convergence of the Hellinger metric $\mu^{\mathfrak y}_\RTE(\sigma)$ and $\mu^{\mathfrak y}_\DE(\sigma)$ to zero.
\end{itemize}

Before these three pieces of analysis, recall that
the forward measurement operators for the RTE and DE
are defined in \eqref{eq:GRTE} and \eqref{eq:GDE} respectively and
that Bayes' theorem~\eqref{bayestheorem} delivers the formulae for
the posterior distributions in \eqref{eq:bayes1}--\eqref{eq:bayes3}. 

\subsection{Convergence Of The Forward Map}
\label{cofm}
For simplicity we assume that $l_j$ is the 
linear funtional that takes corresponding to evaluation at 
point $x_j\in\partial\Omega$: 
\begin{equation}
l_j(f) = f(x_j)\,,\quad\text{where}\quad x_j\in\partial\Omega\,.\label{ptwise}
\end{equation}
Other linear functionals can be handled with similar analysis. Then
\begin{equation}\label{eqn:calG_RTE_jk}
\mathcal{G}^{\RTE}_{jk}(\sigma) = l_j(\mathcal{H}^\RTE(\sigma)\phi_k) = -\frac{1}{C_d\Kn}\int_{\Gamma(x_j)} (v\cdot n)f(x_j,v)\rd{v}\,,
\end{equation}
and
\begin{equation}\label{eqn:calG_DE_jk}
\mathcal{G}^{\DE}_{jk}(\sigma) = l_j(\mathcal{H}^\DE(\sigma)\xi_k) = \frac{1}{\sigma(x_j)}\frac{\partial\rho}{\partial n_{x_j}}(x_j)\,.
\end{equation}
where $f$ and $\rho$ are the solutions to the RTE and the DE with $\phi_k$ and $\xi_k$ as incoming conditions, respectively.

We now have the following proposition:
\begin{prop}\label{prop:conv_G_nonlinear}
Assume that
$\phi_k(x,v) = \xi(x) -\Kn \frac{1}{\sigma(x)}v(x)\cdot\nabla\xi_k(x)$. Then,
under Assumption~\ref{ass:sigma}, the forward maps $\mathcal G^\RTE$ and $\mathcal G^\DE$ satisfy
\begin{equation}\label{eqn:uniform_bound_diff}
\sup_{\sigma \in \mathcal{A}}\|\mathcal{G}^\RTE(\sigma) - \mathcal{G}^\DE(\sigma) \|_{\infty} \leq \frac{C_\mathcal{A}}{C_d} \Kn\,.
\end{equation}
Furthermore, there is a constant $C$ that only depends on $C_1$ and $\Omega$ so that
\begin{equation}\label{uniform_bound}
\max\Big\{\sup_{\sigma \in \mathcal{A}}\|\mathcal{G}^\RTE(\sigma)\|_\infty\,, \sup_{\sigma \in \mathcal{A}}\|\mathcal{G}^\DE(\sigma) \|_{\infty}\Big\} \leq C\,.
\end{equation}
\end{prop}
\begin{proof}
To show the first item it is enough to prove that for every $j$ and $k$,
\[
|\mathcal{G}^\RTE_{jk}(\sigma) - \mathcal{G}_{jk}^\DE(\sigma)| \leq \frac{C_\mathcal{A}}{C_d} \Kn\,.
\]
Note that, for any $y$, ${\mathbb S}^{d-1}\backslash\Gamma(y)$ is the set on
which $(v \cdot n)=0$. Hence, employing
\eqref{eq:needthis}, \eqref{eqn:calG_RTE_jk} and~\eqref{eqn:calG_DE_jk},
and defining
$$r=f-\rho +\frac{\Kn}{\sigma}v\cdot \nabla\rho,$$
we then have
\begin{align}\label{gdiff}
|\mathcal{G}^\RTE_{jk}(\sigma) - \mathcal{G}_{jk}^\DE(\sigma)| & =\left| \frac{1}{\sigma} \frac{\partial \rho}{\partial n} (x_j)+ \frac{1}{C_d \Kn} \int_{\Gamma(x_j)} v\cdot n f \rd v \right|\\
& =\left| \frac{1}{\sigma} \frac{\partial \rho}{\partial n}(x_j) + \frac{1}{C_d \Kn} \int_{{\mathbb S}^{d-1}} (v\cdot n) (\rho-\frac{\Kn}{\sigma}v\cdot\nabla\rho + r) \rd v \right|\nonumber\\
& =\left| \frac{1}{\sigma} \frac{\partial \rho}{\partial n}(x_j) + \int_{{\mathbb S}^{d-1}} \frac{1}{C_d}\left[\frac{-1}{\sigma}(v\cdot n)(v\cdot\nabla\rho)+ \frac{r}{\Kn}\right] \rd v \right|\nonumber\\
& \leq\left| \frac{1}{\sigma} \frac{\partial \rho}{\partial n}(x_j) - \frac{1}{\sigma} \frac{\partial \rho}{\partial n}(x_j)\right|+\frac{C_\mathcal{A}}{C_d}\Kn  = \frac{C_\mathcal{A}}{C_d}\Kn\,.\nonumber
\end{align}
Here we used Theorem~\ref{thm:convergence} which states
\begin{equation*}
\|r\|_{L_\infty(\Omega\times\mathbb{S}^{d-1})}=\left\|f - \left(\rho - \frac{\Kn}{\sigma}v\cdot\nabla\rho\right)\right\|_{L_\infty(\Omega\times\mathbb{S}^{d-1})} \leq C_\mathcal{A}\Kn^2\,,
\end{equation*}
which concludes the proof of equation~\eqref{eqn:uniform_bound_diff}. Equation~\eqref{uniform_bound} is a direct consequence of Proposition~\ref{prop:uniform_bound_Neumann}
and by combining the conclusion of 
Proposition~\ref{prop:uniform_bound_Neumann} with equation~\eqref{eqn:uniform_bound_diff}.
\end{proof}

\subsection{Convergence In Kullback-Leibler Divergence} \label{sec:KL}
We use the convergence of the forward map to show the convergence in the posterior distribution using the Kullback-Leibler divergence. 

\begin{thm} \label{t:KL1}
Let the assumptions of Proposition \ref{prop:conv_G_nonlinear}, together
with Assumption \ref{a:add}, hold. Then
\[
\KL(\mu^{\mathfrak y}_\RTE, \mu^{\mathfrak y}_\DE)\leq \mathcal{O}(\Kn)\,.
\]
\end{thm}
\begin{proof}
We first note that, over the set $\mathcal {A}$ which contains the
support of the common prior measure $\mu_0$, the likelihoods
$\mu^{\sigma}_\RTE(\mathfrak y)$ and $\mu^{\sigma}_\DE(\mathfrak y)$ are bounded
uniformly from above and below. Hence the measures  
$\mu^{\mathfrak y}_\RTE$ and $\mu^{\mathfrak y}_\DE$ are mutually absolutely
continuous and have densities with respect to one another.
In particular we may define
\begin{equation}\label{eqn:KL}
\KL(\mu^{\mathfrak y}_\RTE, \mu^{\mathfrak y}_\DE) = \int_{\mathcal{A}} \Bigl(\log \frac{d\mu^{\mathfrak y}_\RTE}{d\mu^{\mathfrak y}_\DE}(\sigma)\Bigr)\rd\mu^{\mathfrak y}_\DE(\sigma)\,,
\end{equation}
where $\sigma\in \mathcal{A}$. Clearly $\mu^{\mathfrak y}_\DE$ has no $\Kn$ 
dependence, and so it suffices to show that $\log \frac{d\mu^{\mathfrak y}_\RTE}{d\mu^{\mathfrak y}_\DE}$ is $\mathcal O (\Kn)$, uniformly on $\mathcal{A}$.
Using \eqref{eq:bayes1}--\eqref{eq:bayes3}, we find
\[
\log \frac{d\mu^{\mathfrak y}_\RTE}{d\mu^{\mathfrak y}_\DE}(\sigma) = \log \left( \frac{\mu_0(\sigma)\mu^{\sigma}_\RTE(\mathfrak y)}{Z^\RTE} \frac{Z^\DE}{\mu_0(\sigma)\mu_\DE^{\sigma}(\mathfrak y)}\right)  = \log \frac{Z^\DE}{Z^\RTE} + \log \frac{\mu^{\sigma}_\RTE(\mathfrak y)}{\mu^{\sigma}_\DE(\mathfrak y)}\,.
\]
We will show that both terms are $\mathcal O(\Kn)$. 
Recalling~\eqref{eq:bayes2} we see that
\begin{align*}
\big| \mu^{\sigma}_\RTE(\mathfrak y)-\mu^{\sigma}_\DE(\mathfrak y)|
=&~ \Bigg| \exp \left( - \frac{\| \mathfrak y - \mathcal G^\RTE(\sigma)\|^2}{2\gamma^2}\right) - \exp \left( -\frac{\| \mathfrak y - \mathcal G^\DE(\sigma)\|^2}{2\gamma^2}\right)\Bigg| \nonumber\\
 \leq &~ c\Big| \|\mathfrak y - \mathcal G^\RTE(\sigma)\|^2 - \| \mathfrak y-\mathcal G^\DE(\sigma)\|^2\Big|\,,
\end{align*}
where $c<\infty$ is the Lipschitz constant for $\exp(-|x|/2\gamma^2).$
Now note that
\[
 \|\mathfrak y - \mathcal G^\RTE(\sigma)\|^2 - \| \mathfrak y-\mathcal G^\DE(\sigma)\|^2=-\left(2\mathfrak y - \mathcal G^\RTE(\sigma) - \mathcal G^\DE(\sigma)\right)^\top\left(\mathcal G^\RTE(\sigma)-\mathcal G^\DE(\sigma)\right)\,,
\]
and according to Proposition \ref{prop:conv_G_nonlinear},
\begin{equation}\label{eqn:lip}
\sup_{\sigma \in \mathcal{A}}\|c\left(2\mathfrak y - \mathcal G^\RTE(\sigma) - \mathcal G^\DE(\sigma)\right)\|_\infty < \infty\,,
\end{equation}
we deduce that 
\[
\sup_{\sigma \in \mathcal{A}}\big|  \mu^{\sigma}_\RTE(\mathfrak y)-\mu^{\sigma}_\DE(\mathfrak y)\big| = \mathcal{O}(\Kn)\,.
\]
%
Using the definition of the normalization factor and noting that $\int\rd\mu_0(\mathcal{A}) = 1$, we also have
\[
\big|Z^\RTE -Z^\DE \big| \leq \int_{\mathcal A} \big|\mu^{\sigma}_\RTE(\mathfrak y) -\mu^{\sigma}_\DE(\mathfrak y) \big| \rd\mu_0(\sigma) =\mathcal{O}(\Kn)\,.
\]
Noting that $Z^\DE$ and $\mu^{\sigma}_\DE(\mathfrak y)$ are bounded from
below uniformly with respect to $\sigma \in \mathcal{A}$, we deduce 
from the two preceding displays that
$$\sup_{\sigma \in {\mathcal A}}\Big|\log \frac{Z^\DE}{Z^\RTE} + \log \frac{\mu^{\sigma}_\RTE(\mathfrak y)}{\mu^{\sigma}_\DE(\mathfrak y)}\Big|=\mathcal{O}(\Kn)\,$$
which completes the proof.
\end{proof}

\subsection{Convergence In Hellinger Metric}\label{sec:Hellinger}
Convergence in the Hellinger metric has a very similar proof  to that used
in KL divergence.

\begin{thm}\label{thm:hellinger_nonlinear}
Let the assumptions of Proposition \ref{prop:conv_G_nonlinear}, together
with Assumption \ref{a:add}, hold. Then
\[
d_\text{Hell}(\mu^{\mathfrak y}_\RTE,\mu^{\mathfrak y}_\DE) \leq \mathcal{O}(\Kn)\,.
\]
\end{thm}
\begin{proof}
We first recall the definition of the Hellinger distance between two distributions in section \ref{bayes}, using $\lambda=\mu_0$ as the reference measure:
\begin{align*}
d_\text{Hell}(\mu^{\mathfrak y}_\RTE,\mu^{\mathfrak y}_\DE)^2 = \frac{1}{2} \int_{\mathcal A} \left( \sqrt{\frac{d\mu^{\mathfrak y}_\RTE}{d\mu_0}}(\sigma) - \sqrt{\frac{d\mu^{\mathfrak y}_\DE}{d\mu_0}}(\sigma)\right)^2 \mu_0(\rd \sigma)\,.
\end{align*}
Following~\cite{stuart} we obtain
\begin{align}\label{ineqn:hellinger}
d_\text{Hell}(\mu^{\mathfrak y}_\RTE,\mu^{\mathfrak y}_\DE)^2 &=\frac{1}{2} \int_{\mathcal A} \left[ \frac{1}{\sqrt{Z^\RTE}} \exp \left(\frac{-1}{2\gamma^2} \| \mathfrak y - \mathcal G^\RTE(\sigma)\|^2_2 \right) - \frac{1}{\sqrt{Z^\DE}}\exp \left( \frac{-1}{2\gamma^2} \| \mathfrak y - \mathcal G^\DE(\sigma)\|^2_2 \right)\right]^2 \rd \mu_0\\  \nonumber
&\leq I_1 + I_2\,,
\end{align}
where
\[
I_1  = \frac{1}{Z^\RTE} \int_{\mathcal A} \left[ \exp \left( - \frac{1}{2\gamma^2} \| \mathfrak y - \mathcal G^\RTE(\sigma)\|^2 \right) - \exp \left( -\frac{1}{2\gamma^2} \| \mathfrak y - \mathcal G^\DE(\sigma)\|^2\right) \right]^2  \rd\mu_0(\sigma)\,,
\]
and
\[
I_2  = \left| (Z^\RTE)^{-1/2} - (Z^\DE)^{-1/2}\right|_2^2 \int_{\mathcal A} \exp\left( -\frac{1}{2\gamma^2}\| \mathfrak y - \mathcal G^\DE(\sigma)\|_2^2\right) \rd\mu_0(\sigma)\,.
\]
With the same argument, we have
\begin{align}\label{eqn:nonlinear_I1}
I_1 & \leq \frac{c}{Z^\RTE} \int_{\mathcal A} \| \mathcal G^\RTE - \mathcal G^\DE\|^2_2\,\|\mathcal G^\RTE+\mathcal G^\DE-2\mathfrak y\|_2^2\,\rd \mu_0 = \mathcal{O}(\Kn^2)\,,
\end{align}
where we have used
\[
\| \mathcal G^\RTE - \mathcal G^\DE \|_\infty<C_\mathcal{A}\Kn/C_d,
\]
and the Lipschitz argument as in~\eqref{eqn:lip}. To deal with $I_2$, we notice that
\begin{equation}\label{ineqn:I_2}
I_2  \leq \max\left\{ (Z^\RTE)^{-3},(Z^\DE)^{-3}\right\}\left| Z^\RTE - Z^\DE\right|^2 \int_{\mathcal A} \exp\left( - \frac{1}{2} \| \mathfrak y - \mathcal G^\DE(\sigma)\|^2_2 \right) \rd \mu_0(\sigma)\,,
\end{equation}
using the fact that
\[
|(Z^\RTE)^{-1/2} - (Z^\DE)^{-1/2}|^2 \leq \max\{(Z^\RTE)^{-3},(Z^\DE)^{-3}\}|Z^\RTE - Z^\DE|^2.
\]
According to the definition of $Z^{\RTE,\DE}$, we have
\begin{align*}
|Z^\RTE - Z^\DE| &\leq \int_{\mathcal A} \left|\exp\left( -\frac{1}{\gamma^{2}} \| \mathfrak y - \mathcal G(\sigma)^\RTE\|^2_2\right) - \exp\left( - \frac{1}{\gamma^{2}} \| \mathfrak y - \mathcal G(\sigma)^\DE\|^2_2\right)\right| \rd \mu_0(\sigma)\\
&\leq c\int_{\mathcal A} \Big| \| \mathfrak y - \mathcal G^\RTE(\sigma)\|^2_2-\|\mathfrak y - \mathcal G^\DE(\sigma)\|^2_2 \Big| \rd \mu_0(\sigma)\\
&\leq c\int_{\mathcal A}  \| \mathcal G^\RTE - \mathcal G^\DE\|_2\|\mathcal G^\RTE+\mathcal G^\DE-2\mathfrak y\|_2\rd\mu_0(\sigma)\\
&=C_\mathcal{A}\Kn/C_d\,,
\end{align*}
where we used~\eqref{eqn:lip}. Plugging these back in~\eqref{ineqn:I_2}, we have
\[
I_2 =\mathcal O(\Kn^2).
\]
Together with the boundedness of $I_1$ and the inequality~\eqref{ineqn:hellinger}, we conclude
\[
\Hell(\mu^{\mathfrak y}_\RTE,\mu^{\mathfrak y}_\DE) = \mathcal O(\Kn)\,.
\]
\end{proof}

\section{Linearized Inverse Problems} 
\label{sec:L}
In this section we study approximations of the two Bayesian
inverse problems in the linearized setting. We show asymptotic closeness 
of the posterior distributions in the small Knudsen number regime $\Kn \ll 1.$
Equations~\eqref{eqn:gammaRTE}--\eqref{eqn:fredhom_rte2} and~\eqref{eqn:fredhom_diff} give rise to the following approximate inverse problems:
\begin{equation}\label{eqn:forward_linear}
\mathfrak y^\RTE_{\text{lin}} = {G}^\RTE(w) + \eta\,,\quad\text{and}\quad \mathfrak y^\DE_{\text{lin}} = {G}^\DE(w) +  \eta\,,
\end{equation}
where
\[
\mathfrak y^\RTE_{\text{lin}}  = \mathfrak y - \mathcal G^{\RTE}(\log(u_0))\,, \quad\text{and}\quad \mathfrak y^\DE_{\text{lin}}  = \mathfrak y - \mathcal G^{\DE}(\log(u_0))\,.
\]
is a vector of length $JK$ and can be regarded as the linearized data. It can be obtained by subtracting $\mathfrak y$, the collected measurements with $\{\phi_k, \, k=1,\dots,K\}$ being the input data and $\{l_j,\, j=1, \dots, J\}$ being the pointwise evaluation operator, as defined in~\eqref{ptwise}, and $\mathcal G^{\RTE,\DE}(\log(u_0))$, the background data that is precomputed using~\eqref{rtelin} or~\eqref{eqn:diff_lin} with the same input and measurement operator, and the background medium $u_0$.

Assuming $\eta\sim\mathcal{N}(0,\gamma^2\mathbb{I})$ as always, now we have the likelihood functions:
\[
\mmu^w_\RTE(\mathfrak y^\RTE) = \exp\left(-\frac{1}{2\gamma^{2}}\|\mathfrak y_\text{lin}^\RTE - {G}^\RTE(w)\|^2_2\right)\quad\text{and}\quad\mmu^w_\DE(\mathfrak y^\DE) = \exp\left(-\frac{1}{2\gamma^{2}}\|\mathfrak y_\text{lin}^\DE - {G}^\DE(w)\|^2_2\right)
\]
The two models use the same prior distribution function $\mmu_0(w)$, satisfying
\begin{equation*}
\int_{\mathcal{C}^{3}(\Omega)}1\rd\mmu_0=1\,.
\end{equation*}
The posterior distributions are then
\begin{equation}
\mmu^{\mathfrak y}_\RTE(dw) = \frac{1}{Z^\RTE}\mmu^{w}_\RTE(\mathfrak y^\RTE)\mmu_0(dw)\,, \quad\text{and}\quad \mmu^{\mathfrak y}_\DE(w) = \frac{1}{Z^\DE}\mmu^{w}_\DE(\mathfrak y^\DE)\mmu_0(dw)\,,
\end{equation}
where $Z^\RTE$ and $Z^{\DE}$ are the normalization factors.

\subsection{Convergence Of Linearized Forward Map}
\label{colfm}
We first show the convergence of the interpreters $\gamma^{\RTE,\DE}$, which will allow us to show the convergence of the two forward maps.

\begin{prop}\label{prop:conv_linear_gamma} Assume $u_0\in\mathcal{A}_u$, then for $\Kn$ sufficiently small, $\gamma^\RTE \to \gamma^\DE$. More specifically, for every $j$ and $k$,
\begin{equation}
\| \gamma^\RTE_{jk} - \gamma^\DE_{jk} \|_{L_\infty(\Omega)} \leq C \Kn^2\,.
\end{equation}
Here the constant $C$ only depends on $C_\mathcal{A}$ and $C_1$.
\end{prop}
\begin{proof}
Recall the definition of $\gamma_{jk}$ in~\eqref{eqn:gammaRTE}
\begin{align*}
\gamma^\RTE_{jk}(x) = -\frac{e^{u_0}}{C_d\Kn^2} \int_{\mathbb S^{d-1}} g_j(x,v) \mathcal L f_{\lin,k}(x,v)\rd v\,,
\end{align*}
where $g_j$ and $f_{\lin,k}$ solve~\eqref{rteadj} and~\eqref{rtelin} with $\delta_{y_j}$ and $\phi_k$ as boundary conditions. We further recall Theorem~\ref{thm:convergence}, so that we have
\begin{equation*}
\|g_j -\rho_{g_j} -\eps e^{-u_0} v\cdot\nabla\rho_{g_j} \|_{L_\infty(\Omega\times\mathbb{S}^{d-1})} <C_\mathcal{A}\Kn^2\,,\quad\text{and}\quad \|f_{\lin,k} -\rho_{f_k} +\eps e^{-u_0} v\cdot\nabla\rho_{f_k} \|_{L_\infty(\Omega\times\mathbb{S}^{d-1})} <C_\mathcal{A}\Kn^2\,,
\end{equation*}
where $\rho_{g,j}$ and $\rho_{\lin,k}$ solve
\begin{equation*}
-\nabla\cdot(e^{-u_0} \nabla\rho)  = 0 \,,
\end{equation*}
with boundary condition $\delta_{y_j}$ and $\xi_k$ respectively. Recalling $\mathcal{L}\rho = 0$ for all $\rho$, and that $\int_{\mathbb S^{d-1}}v\rd{v} = 0$, then
\[
\|\gamma^\RTE_{jk}(x) -\gamma^\DE_{jk}(x)\|_{L_\infty(\Omega)} =\|\gamma^\RTE_{jk}(x) -e^{-u_0} \nabla\rho_{g,j}\cdot\nabla\rho_{\lin,k} \|_{L_\infty(\Omega)} = \mathcal{O}(\Kn^2)\,.
\]
We conclude the proof since this holds for every $j$ and $k$.

\end{proof}
We emphasize that $\gamma^\RTE$ is uniquely determined by $g$ and $f_\lin$ that solve~\eqref{rteadj} and~\eqref{rtelin}, and that the two equations depend merely on $u_0$. So the convergence holds true as long as $u_0\in\mathcal{A}_u$, and there is no requirement on $w$.

\begin{cor} \label{c:H2}
For any fixed $u_0\in\mathcal{A}_u$, assume $w\in C^3$, if $\Kn$ significantly small, then $G^\RTE \to G^\DE$. More specifically,
\begin{equation}
\| G^\RTE - G^\DE \|_{\infty} \leq C \Kn^2\|w\|_{L_2(\Omega)}\,,
\end{equation}
where $G^{\RTE,\DE}$ are two vectors of length $JK$, and $C$ only depends on $C_\mathcal{A}$, $J$ and $K$.
\end{cor}
\begin{proof}
According to the definition of $G^{\RTE,\DE}$,
\[
G^\RTE_{jk} - G^\DE_{jk} = \langle \gamma^\RTE_{jk} - \gamma^\DE_{jk}\,,w\rangle\leq \|\gamma^\RTE_{jk} - \gamma^\DE_{jk}\|_{L_2(\Omega)}\|w\|_{L_2(\Omega)}\,.
\]
We conclude using the result from Proposition~\ref{prop:conv_linear_gamma}, and that $\|\gamma^\RTE_{jk} - \gamma^\DE_{jk}\|_{L_2(\Omega)}\lesssim\|\gamma^\RTE_{jk} - \gamma^\DE_{jk}\|_{L_\infty(\Omega)}$ for:
\[
\|G^\RTE - G^\DE\|_2 = \sqrt{\sum_{jk}|G^\RTE_{jk} - G^\DE_{jk}|^2} \leq \sqrt{JK}C\Kn^2\|w\|_{L_2(\Omega)}\,.
\]
\end{proof}

\subsection{Convergence In Hellinger Metric}
\label{cihm}
The proof of the following result is a straightforward extension 
of Theorem~\ref{thm:hellinger_nonlinear} and hence we only sketch it.

\begin{thm} \label{t:H2}
Consider the linearized setting with ${u}_0\in\mathcal{A}_u$ and assume
that $\nu_0$ is a centred Guassian measure supported on the
space $C^3(\Omega).$ Then
the Hellinger distance between the posterior distribution 
$\mmu^{\mathfrak y}_\RTE$ and $\mmu^{\mathfrak y}_\DE$ is bounded by 
$\mathcal{O}(\Kn)$: 
\[
d_\text{Hell}(\mmu^{\mathfrak y}_\RTE,\mmu^{\mathfrak y}_\DE) \leq \mathcal{O}(\Kn^2)\,.
\]
\end{thm}

\begin{proof}[Sketch Proof]

The primary difference of this theorem with Theorem~\ref{thm:hellinger_nonlinear} is that the data $\mathfrak y$ is subtracted by $\mathcal{G}^{\RTE}(\log(u_0))$ and $\mathcal{G}^\DE(\log(u_0))$, and that the linear operator can be made explicit: $G^{\RTE,\DE} = \langle\gamma^{\RTE,\DE},w\rangle$. As a result, the estimates for $I_1$ and $I_2$ change accordingly. The proof is rather similar to that for Theorem~\ref{thm:hellinger_nonlinear}, so we omit the details and only estimate $I_1$ here:
\begin{align*}
I_1 & \leq \frac{c}{Z^\RTE} \int_{C^{3}(\Omega)} \| G^\RTE(w) - G^\DE(w)\|^2_2\,\|G^\RTE(w)+ G^\DE(w)-\mathfrak y^\DE_\lin-\mathfrak y^\RTE_\lin\|_2^2\,\rd \nu_0(dw)\\
&\leq C\Kn^4\int_{C^{3}(\Omega)}\|w\|^2_{L_2(\Omega)}\bigl(1+\|w\|^2_{L_2(\Omega)}\bigr)\rd\nu_0\,.
\end{align*}
For the second inequality to hold true, we first use the conclusion from Corollary~\ref{c:H2}, and to bound the second term, we simply use:
\[
G^\RTE(w) \leq\langle\gamma^\DE\,,w\rangle + C\Kn^2\|w\|_{L_2(\Omega)}\,,\quad\mathcal G^{\RTE}(\log(u_0))\leq\mathcal G^{\DE}(\log(u_0))+C\Kn\,,
\]
to obtain
\begin{align*}
\|G^\RTE(w)+ G^\DE(w)-\mathfrak y^\DE_\lin-\mathfrak y^\RTE_\lin\|_2&\leq \|2\langle\gamma^\DE\,,w\rangle - 2\mathcal G^{\DE}(\log(u_0)) -2\mathfrak y\|_2 + C\Kn^2\|w\|_{L_2(\Omega)} +C\Kn\\
&\leq 2\|\mathcal G^{\DE}(\log(u_0)) +\mathfrak y\|_2 +2\|\langle\gamma^\DE\,,w\rangle\|_2+  C\Kn^2\|w\|_{L_2(\Omega)} +C\Kn\\
&\leq C+C\|w\|_{L_2(\Omega)}\,.
\end{align*}

Application of the Fernique theorem \cite{da2014stochastic} shows that
we have $\int_{C^{3}}\|w\|^p_{L_2(\Omega)}\rd\nu_0$ bounded by a constant 
(independent of $\Kn$) for any $p$ and that
\begin{align*}
I_1 \leq C\Kn^4\,.
\end{align*}

The estimate for $I_2$ is very similar, and therefore
\begin{align*}
d_\text{Hell}(\nu^{\mathfrak y}_\RTE,\nu^{\mathfrak y}_\DE)^2 \leq I_1 + I_2 = \mathcal{O}(\Kn^4)\,,
\end{align*}
which leads to the conclusion of Theorem~\ref{t:H2}.
\end{proof} 

Comparing the preceding theorem with Theorem~\ref{thm:hellinger_nonlinear}, a very interesting phenomenon we immediately see is that the convergence in the linearized setting has a higher rate. This higher rate is a direct consequence of the convergence in $\gamma$ in which the $\mathcal{O}(\Kn)$ cancels due to the symmetry between the forward model and the adjoint.


%

\subsection{Implications For Posterior Convergence} \label{ifpc}
In the linear setup, if the prior distribution and the likelihood functions are both Gaussian functions, the posterior distribution is also Gaussian, and is 
thus completely determined by its mean and covariance, or the leading two moments. In our case, $u_0\in\mathcal{A}_u$, and $w\in C^3$, and the prior is 
supported in $C^3(\Omega)$ for $w$. Thus distances between the posterior
distributions computed using the RTE and DE can be estimated from distances
between means and covariances. To this end, consider the following lemma.

\begin{lem}[Lemma 7.14 from~\cite{hellinger}] \label{expbound}
Let $\mmu$ and $\mmu'$ be two probability measures on a separable Banach space $X$. Assume also that $f: X \to E$, where $(E, \| \cdot \|)$ is a separable Banach space, is measurable and has second moments with respect to both $\mmu$ and $\mmu'$. Then
\[
\| \mathbb E^\mmu f - \mathbb E^{\mmu'}f \| \leq 2 \left( \mathbb E^\mmu \| f\|^2 + \mathbb E^{\mmu'} \| f\|^2\right)^{\frac{1}{2}} \Hell(\mmu,\mmu').
\]
Furthermore, if $E$ is a separable Hilbert space and $f: X \to E$ also has fourth moments, then
\[
\| \mathbb E^\mmu (f\otimes f) - \mathbb E^{\mmu'}(f\otimes f) \| \leq 2 \left( \mathbb E^\mmu \|f \|^4 + \mathbb E^{\mmu'} \| f \|^4 \right)^{\frac{1}{2}} \Hell(\mmu,\mmu'),
\]
where the operator norm on $E$ is employed.
\end{lem}

When applied in our case, we obtain:
\begin{cor}
Let $m^{\RTE,\DE}_\post$ and $\mathcal C^{\RTE,\DE}_\post$ denote the mean function and the covariance operator computed from the posterior distribution 
of the radiative transfer and diffusion model in the linearized setting. Then
\[
\| m^\RTE_\post - m^\DE_\post \| \leq \mathcal O(\Kn^2)\,, \quad \| \mathcal C^\RTE_\post - \mathcal C^\DE_\post \| \leq \mathcal O(\Kn^2)\,.
\]
Here the norm for the mean is the standard norm on $C^3(\Omega)$ and the
norm for the covariance is the operator norm on $H^3(\Omega).$
\end{cor}

\begin{proof} Let $f$ as in the statement of Lemma \ref{expbound} be the identity map, and the spaces $X$ and $E$ be $C^3(\Omega)$ equipped with $L_2$ norm, then $f(w) = w$. Let $\mmu$ and $\mmu'$ be the posterior distributions $\mmu^{\mathfrak y}_\RTE$ and $\mmu^{\mathfrak y}_\DE$ respectively. Then:
\[
\| m^\RTE_\post - m^\DE_\post \|_{L_2(\Omega)} =\| \mathbb E^{\mmu^{\mathfrak y}_\RTE}w-\mathbb E^{\mmu^{\mathfrak y}_\DE}w\|_{L_2(\Omega)} \leq 2 \left( \expRTEm \| w \|^2_{L_2(\Omega)} + \expDEm\| w \|^2_{L_2(\Omega)} \right)^{\frac{1}{2}} \Hell(\mmu^{\mathfrak y}_\RTE, \mmu^{\mathfrak y}_\DE)\,.
\]
Since
\[
\expRTEm \| f \|_{L_2(\Omega)}^2 = \int_{C^3(\Omega)} \|w\|^2_{L_2(\Omega)}\rd\mmu^{\mathfrak y}_{\RTE}\lesssim \int_{C^3(\Omega)} \|w\|^2_{L_2(\Omega)}\rd\mmu_0<C
\]
and that
\[
\expDEm \| f \|^2 = \int_{\mathcal A} \|w\|_{L_2(\Omega)}^2\rd\mmu^{\mathfrak y}_{\DE}\lesssim  \int_{C^3(\Omega)} \|w\|^2_{L_2(\Omega)}\rd\mmu_0<C\,,
\]
where we have again used $\int_{\mathcal C^3(\Omega)} \|w\|^2_{L_2(\Omega)}\rd\mmu^{\mathfrak y}_{\RTE} <C$ using the Fernique theorem, and that $\rd\mmu^{\mathfrak y}_{\RTE,\DE}\lesssim\rd\mmu_0$~\cite{hellinger}, we achieve the convergence of the mean function by applying Theorem~\ref{t:H2}. The same analysis is applied to analyze the covariance.
\end{proof}

%

\section{Conclusion} \label{sec:C}
In this paper, we study the inverse problem of diffuse optical tomography 
to reconstruct the scattering coefficient. Partial and noisy data is
assumed, and hence a Bayesian formulation of inversion is natural. 
Two forward models are described, one employing the radiative transfer 
equation and the other employing the diffusion equation respectively. 
Multiscale analysis demonstrates that solutions of the two forward models are close in the optically thick (small Knudsen
number) regime, and this allows us to quantify the convergence of the two 
Bayesian inverse problems. In particular, we show that $\mu^{\mathfrak y}_\RTE$ and $\mu^{\mathfrak y}_\DE$, the two posterior distribution functions, are $\mathcal{O}(\Kn)$ away from each other in both Kullback-Leibler divergence sense, and the Hellinger sense, for both nonlinear and linear cases.
Forward solution of the diffusion equation is computationally less burdensome
than the radiative transfer equation, and the theory justifies using it
within Bayesian inversion algorithms where multiple forward model
evaluations are required.
We have employed a setting in which compatible boundary conditions
are used for the two forward models. It would also be of interest to
study extensions of this. However the primary stumbling block here
is the analysis of the forward problem itself.
The approach developed in this paper will 
apply to other Bayesian inverse problems
whose forward problems are close, and can be used to justify inversion
algorithms which employ an averaged or (as in this case) homogenized
approximate forward model, in order to speed-up computation.

\vspace{0.1in}

\noindent{\bf Acknowledgements}
AMS is supported by US AFOSR grant FA9550-17-1-0185. KN and QL are supported by NSF DMS 1619778, 1750488 and NSF TRIPODS 1740707.

\vspace{0.1in}

\bibliography{DOTBFCitations}
\bibliographystyle{plain}

\section{Appendix}
We give the rigorous proof here for Theorem \ref{thm:convergence}. The two statements are proved in the same way, and for generality, we will only prove the second one, and the proof for the first statement, or even for higher order expansions, are easy extensions.
\begin{proof}
We repeat the RTE with a specially designed boundary condition,
\begin{equation*}
\begin{cases}
v\cdot\nabla f = \frac{\sigma}{\Kn}\mathcal Lf\\
f|_{\Gamma_-} = \xi(x) -\frac{\Kn}{\sigma} v\cdot\nabla\rho(x)
\end{cases}\,,
\end{equation*}
where $\rho$ satisfies
\begin{equation*}
\nabla\cdot(\sigma^{-1}\nabla\rho) = 0\quad x\in\Omega\,,\quad\text{with}\quad \rho|_{\partial\Omega} = \xi(x)\,.
\end{equation*}
Now we decompose the solution to the RTE as
\begin{equation}
f = f_0+\Kn f_1 + \Kn^2 f_2 + f_r\,,\nonumber
\end{equation}
where $f_0 = \rho$, $f_1 = -\frac{1}{\sigma}v\cdot\nabla\rho$, and $f_2 = \frac{1}{\sigma}\mathcal{L}^{-1}\left[(v\cdot \nabla)\frac{1}{\sigma}(v\cdot\nabla)\rho\right]$. Note that for the definition of $f_2$ to hold true, it is necessary
that 
\begin{equation*}
(v\cdot \nabla)\frac{1}{\sigma}(v\cdot\nabla)\rho\in\text{Range}\,\mathcal{L}\,,
\end{equation*}
which in turn requires
\[
\langle (v\cdot \nabla)\frac{1}{\sigma}(v\cdot\nabla)\rho\rangle_v = C_d\nabla\cdot(\frac{1}{\sigma}\nabla\rho) = 0\,.
\]

Since $\rho$ is $\frac{1}{\sigma}$-harmonic with smooth boundary condition $\|\xi\|_{L_\infty(\partial\Omega)}<C_1$, then by the maximum principle~\cite{gilbarg1983},
\begin{equation*}
\|\rho\|_{L_\infty(\Omega)}<C_1\,,\quad\text{and}\quad \|\partial_{i}\rho\|_{L_\infty(\Omega)}<C_2\quad\text{and}\quad \|\partial_{ij}\rho\|_{L_\infty(\Omega)}<C_3\,.
\end{equation*}
Then since $\mathcal{L}^{-1}$ is a bounded operator on $\NullL^\perp$, we know that both $f_1$ and $f_2$ are bounded, meaning there is a constant $C_4$ that depends on $C_1$, $C_2$ and $C_3$:
\begin{equation*}
\|f_1\|_{\infty} = \|\sigma^{-1}v\cdot\nabla\rho\|_{L_\infty(\Omega)} = \|\sigma^{-1}\|_{L_\infty(\Omega)}\|\partial_i\rho\|_{L_\infty(\Omega)} <C_4\,,
\end{equation*}
and
\[
\|f_2\|_{\infty} = \|\sigma^{-1}\mathcal{L}^{-1}\left[(v\cdot\nabla)\sigma^{-1}(v\cdot\nabla)\rho\right]\|_{L_\infty(\Omega)} \leq \|\sigma^{-1}(v\cdot\nabla)\sigma^{-1}(v\cdot\nabla)\rho\|_{L_\infty(\Omega)} < C_4\,,
\]
where we used the boundedness of $\|\sigma^{-1}\|_{L_\infty(\Omega)}<C_1$, $\|\nabla\left(\frac{1}{\sigma}\right)\|_{L_\infty(\Omega)}<C_1$, and the boundedness of the harmonic function and its derivatives.

 Plugging it back into the equation, we have
\[
v\cdot\nabla \left(\rho -\Kn \frac{1}{\sigma}v\cdot\nabla\rho +\Kn^2 f_2 + f_r\right) = \frac{\sigma}{\Kn}\mathcal{L}\left(\rho -\Kn \frac{1}{\sigma}v\cdot\nabla\rho +\frac{\Kn^2}{\sigma} \mathcal{L}^{-1}\left[(v\cdot \nabla)\frac{1}{\sigma}(v\cdot\nabla)\rho\right] + f_r\right)\,.
\]
Since $\rho$ is a constant in $v$, and is thus in $\NullL$, then $\mathcal{L}\rho = 0$. Using the definition of $\mathcal{L}$, we also have $\mathcal{L}(v\cdot \nabla \rho) = -v\cdot\nabla\rho$. Now we cancel the terms and obtain the following equation for $f_r$,
\[
 v\cdot\nabla f_r = \frac{\sigma}{\Kn}\mathcal{L}f_r- \Kn^2 v\cdot\nabla f_2\,.
\]
It is immediate that $f_r$ satisfies RTE and is equipped with a source term of $\mathcal{O}(\Kn^2)$. The boundary condition for $f_r$ is of the same order,
\[
f_r|_{\Gamma_-} = -\Kn^2f_2|_{\Gamma_-}\,.
\]
By maximum principle, the solution is bounded in $L_\infty$ by the boundary condition and the source term, so we have
\[
\|f_r\|_{L_\infty(\Omega)} \leq C_5\|\Kn^2 v\cdot\nabla f_2\|_{L_\infty(\Omega)}+\|\Kn^2f_2\|_{L_\infty(\Gamma)} = \mathcal{O}(\Kn^2)\,,
\]
where $C_5\leq e^{C_1l}$ and $l$ is the longest radius of the domain. This leads to the fact that:
\begin{equation}
\|f - \rho+\Kn v\cdot\nabla\rho\|_{L_\infty(\Omega)} =\|\Kn^2f_2+f_r\|_{L_\infty(\Omega)}\mathcal{O}(\Kn^2)\,.\nonumber
\end{equation}
We note again the constant merely depends on the boundedness of $C_1$, the upper bound of the infinite norm of $\frac{1}{\sigma}$, its derivative and $\xi$.
\end{proof}

\end{document}